\documentclass
{amsart}
\usepackage{amsmath,amsthm,amsfonts,amssymb,mathtools}
\usepackage[all]{xy}
\usepackage{graphicx}
\usepackage{caption}
\usepackage{subcaption}
\usepackage{blindtext}
\usepackage{chngcntr}
\usepackage{amscd}
\usepackage{graphicx}
\usepackage{pb-diagram}
\usepackage{color}
\usepackage{enumitem}

\usepackage[hidelinks]{hyperref}

\hypersetup{
colorlinks=true,       
linkcolor=blue,          
citecolor=red,        
filecolor=magenta,      
urlcolor=cyan           
}

\makeatletter
\newcommand{\doublewidetilde}[1]{{%
  \mathpalette\double@widetilde{#1}%
}}
\newcommand{\double@widetilde}[2]{%
  \sbox\z@{$\m@th#1\widetilde{#2}$}%
  \ht\z@=.9\ht\z@
  \widetilde{\box\z@}%
}
\makeatother

\numberwithin{equation}{section}

\newtheorem{theorem}{Theorem}[section]
\newtheorem{definition}[theorem]{Definition}

\newtheorem{lemma}[theorem]{Lemma}
\newtheorem{remark}[theorem]{Remark}
\newtheorem{prop}[theorem]{Proposition}

\newtheoremstyle{case}{}{}{}{}{}{:}{ }{}
\theoremstyle{case}

\counterwithin*{case}{section}

\def\R{{\mathbb R}}

\def\C{{\mathbb C}}

\newcommand{\BL}{{\mathrm{B}{\mathrm{L}}}}

\newcommand{\rp}{{\phi_{1}}}
\newcommand{\ip}{{\phi_{2}}}
\newcommand{\nphi}{\mathbf n^\phi}
\newcommand{\sspan}{\mathrm{span}}
\newcommand{\mfq}{\mathfrak q}
\newcommand{\mbq}{\mathbf Q}

\newcommand{\cD}{\overline D} 
\newcommand{\ce}{\mathcal E}
\newcommand{\cqq}{C_{\mfq}^{\, \mbq}}
\newcommand{\dist}{\mathrm{dist }}
\newcommand{\ve}{\varepsilon}

\begin{document}

\title{Restriction estimates to complex hypersurfaces}
\author[Juyoung Lee]{Juyoung Lee}
\author[Sanghyuk Lee]{Sanghyuk Lee}

\address{Department of Mathematical Sciences and RIM, Seoul National University, Seoul 08826, Republic of  Korea}
\email{shklee@snu.ac.kr}
\email{ljy219@snu.ac.kr } 


\maketitle
\begin{abstract} 
The restriction problem is better understood for hypersurfaces and  recent progresses  have been made  by  bilinear and multilinear approaches and most recently  polynomial partitioning method which is combined with
those estimates.  However, for surfaces with codimension bigger than 1,  bilinear and multilinear generalization of  restriction estimates 
are  more involved and effectiveness of these multilinear estimates is not so well understood yet. Regarding the restriction problem for the surfaces with codimensions bigger than 1,  the current state of the art is still  at the level of $TT^*$ method which is known to be useful for obtaining  $L^q$--$L^2$ restriction estimates. 
 In this paper, we consider a special type of codimension 2 surfaces which are given by graphs of complex analytic functions and attempt to 
 make progress beyond the  $L^2$ restriction estimates. 
\end{abstract}

\section{Introduction}

Let $S$ be  a smooth compact submanifold $\R^n$($n\leq 3$)  with the usual surface measure  $d\sigma$ (the induced Lebesgue measure) on $S$. 
The $L^p(\mathbb R^n)-L^q(d\sigma)$ boundedness of the restriction operator $f\to \widehat{f}|_S$ has been extensively studied since  the restriction phenomena was first observed by Stein in  the late 1960s. 
As is standard in literature nowadays,  it is more convenient to work with
the dual operator  $\widehat{fd\sigma}$ which is called \emph{extension operator}. 
The ultimate goal of the restriction problem is to characterize  $L^p(S)$--$L^q(\R^n)$   boundedness of $f\to  \widehat{fd\sigma}$  in terms of geometric features of underlying (sub)-manifold $S$. 
Particularly,  when $S$ is the sphere, it was conjectured that $\widehat{fd\sigma}$ should map $L^p(S)$ boundedly to $L^q(\R^n)$ if and only if  $q\geq \frac{p'(n+1)}{n-1}$ and $q>\frac{2n}{n-1}$. 
There is a large body of  literature which is  devoted to this problem. Over the last couple of decades, the bilinear and multilinear approaches have been proved to be most effective, and via these new methods substantial progress has been made. Recently, polynomial partitioning method gave currently the best known result. We refer the reader to \cite{BCT}, \cite{BG}, \cite{G1} for the latest developments.

Restriction phenomena to submanifolds other than hypersurfaces were also studied by some authors. Roughly, we may say that restriction problem is better understood when the dimension of manifold is 1 or its codimension is one. 
When the dimension of $S$ is 1(namely, $S$ is a curve), the restriction estimate is by now fairly well understood \cite{BOS1}, \cite{BOS2}, \cite{BOS3}. However, not much is known regarding the restriction estimate to surfaces of  the intermediate dimensions, that is to say,  when the codimension $k$ of the manifold is between $1$ and $n-1$. The restriction problem for surfaces with codimension $k=2$ was first studied by Christ\cite{C} and later Mockenhaupt\cite{M} with general $k$. 
For certain types of surfaces they established the optimal $L^q \rightarrow L^2$ restriction estimates which can be regarded as an extension of  the Stein-Tomas theorem which concerns  
 the hypersurfaces with nonvanishing Gaussian curvature.  There are some results \cite{BL}, \cite{Ob} beyond  the $L^q \rightarrow L^2$ restriction estimate for the surfaces with intermediate codimensions.  However, it can  be said that,  for most surfaces with codimension between 1 and $n-1$, the current state of the restriction problem is hardly beyond that of the Stein-Tomas theorem in the case of hypersurfaces.

To discuss  the previous results for the restriction estimates for a surface of codimension $k\ge 2$, we consider the extension operator which is given by the surface $(\xi, \Phi(\xi))$. Let $D$ be a bounded region in $\R^d$ and $\Phi : D \rightarrow \R^k$ be a smooth function. Discarding harmless factor associated to parametrization of the surface measure,  it is enough to consider  the operator $E$
which is defined by
$$Ef(x,t)=\int_D e^{i\left( x\cdot \xi +t\cdot \Phi(\xi) \right)}f(\xi)d\xi, \quad \left( x,t\right)\in \R^m\times \R^k.$$
Especially, if $\Phi$ is given  by nontrival quadratic forms,  the optimal $L^2-L^q$ boundedness of $E$ is well understood. In this case, using a Knapp-type example, it is easy to see that $E$ may be bounded from $L^p$ to $L^q$ only if $\frac{d+2k}{q}\leq d( 1-\frac{1}{p} )$.\footnote{This is also true if $Q$ has not trivial second order term at a point in $D$.}  Hence, the best possible $L^2-L^q$ bound is that with $q=\frac{2(d+2k)}{d}$.  This was shown to be true when  $\Phi$ satisfies a suitable curvature condition  \cite{C, M}.

In this paper, we are concerned with restriction estimates for special type of surfaces of codimension 2 which are given by graphs of holomorphic functions. Identifying the complex number  with a point in $\R^2$, complex hypersurface in $\C^n$ can be considered as a  manifold of  codimension 2 in $\R^{2n}$. We begin with introducing some notations. Let us set 
$$\widetilde w\odot w=\Re\Big(\sum_{j=1}^n \widetilde w_j\overline{w_j}\Big), $$
where $\widetilde w=\left( \widetilde w_1, \cdots , \widetilde w_n \right)$, $w=\left( w_1, \cdots , w_n \right)\in \mathbb C^n$. 
For a  function $\phi: \mathbb C^{n-1}\to \mathbb C$ and a bounded measurable set $D$ we define 
$$\mathcal E_D^\phi f(w)=\int_D e^{i w\odot(z,\phi(z))}f(z) dz, \quad w\in \mathbb C^{n},$$
where  $dz$ denotes the usual Lebesgue measure on $(2d-2)$-dimensional Euclidean space. 

If we write $z=(z_1, \dots, z_{n-1})$ and $z_j=x_j+iy_j$, $1\le j \le n-1$,  the above operator is the extension operator given by 
the surface $(x_1,y_1,\dots,x_{n-1}, y_{n-1}, \Re \phi(x+iy), \Im \phi(x+iy))$.
We define the (complex) Hessian of $\phi$ by
\[
H\phi:=\Big [ 
{\partial_{z_i}\partial_{z_j}}\phi\Big]_{1\le i, j\le n-1}.
\]
The natural non-degeneracy condition on $\phi$ is that $H\phi$ has nonzero determinant on $\overline D$.  
It is not difficult to see  that  $|\mathcal E_{\mathbb C^{n-1}}^\phi(\chi)(w)|\lesssim  |w_n|^{-(n-1)}$ whenever $\phi$ is holomorphic and $\det(H\phi)\neq 0$ on the support of $\chi\in C_c^\infty$ (see Lemma \ref{decay_est}). Thus by the 
$TT^*$ argument it is not difficult to see  that,  for $q\ge   \frac{2n+2}{n-1}$, 
$\|\mathcal E_D^\phi f\|_{q}\lesssim \|f\|_2$.
The main result of this paper is to show that  this can be improved when 
$n$ is even.

\begin{theorem}\label{main}
Let $n$ be an even number $\ge 4$ and $D$ be a bounded region in $\C^{n-1}$. 
 Suppose  $\phi(z)$ is holomorphic on $\cD$ and  $\det H\phi\neq 0$ on $\cD$. 
Then,  for $p>\frac{2(n+2)}{n}$, $\Vert \mathcal E_D^\phi f \Vert_{p}\leq C_p \Vert f \Vert_{p}$   for some constant $C_p$.
\end{theorem}
 
When $n=2$, the result  on the optimal range of $p,q$  was obtained by Christ, see \cite[Theorem 3.2]{C1}.  %
Our result relies on the multilinear restriction estimates for the complex surfaces and the induction arugument due to Bourgain  and Guth \cite{BG}. When $n$ is odd,  as it was shown for the quadratic surfaces with principal curvatures of different signs  \cite{BG}, the induction argument based on multilinear restriction estimate is not enough to give estimate with the exponent 
$q<\frac{2n+2}{n-1}$. 
Since  restriction of  the complex analytic surface to subspace  admits subsurface with no curved property,   unlike the case of 
elliptic surfaces  repeated use of multilinear restriction estimate is not allowed.  See Remark \ref{oddrmk}.  However, for $n=3$, the $L^p-L^{q}$ estimates for $p,q$ satisfying $1/p+2/q<1$ and $q>10/3$ were obtained in \cite{BLL}. Our result doesn't recover this and it is a manifestation that  multilinear strategy has certain inefficiency in capturing the curvature property of the underlying surface. 
 
In our proof of Theorem \ref{main},  the holomorphic assumption plays an important role. The assumption not only  makes it possible to  describe transversality condition in a simpler way but also provides good decay property of the Fourier transform of the surface measure.  Though  general forms of multilinear restriction estimates \cite{BBFL, Z} are known, 
not all of  the restriction estimates we need for our purpose appear in literature. In Section 3 we prove these restriction estimates by following the argument in \cite{BBFL} and  making use of  general multilinear Kakeya  and induction  argument.   In section 4, we prove  our main result and discuss about surfaces given by almost complex structure.

\section{Preliminaries}
In this section we review the known multilinear Kakeya and multilinear restriction estimates on which  our results are to be based. 
These estimates generalize multilinear restriction estimates for hypersurfaces \cite{BCT}. In fact, fairly general forms of these estimates can be found in \cite{BBFL}. 
Before stating their results, we introduce some notations and give the statement of  the Brascamp-Lieb inequality, which we also use later. 

\begin{theorem}[Brascamp-Lieb inequality, \cite{BCCT}]\label{BL}
Let $L_j : \R^d \rightarrow \R^{d_j}$ be linear, onto, and $p_j\geq 0$ for $1\leq j \leq m$. Then, 
\begin{equation}
\label{BL}
\int_{\R^d}\prod_{j=1}^m (f_j \circ L_j)^{p_j} \leq C\prod_{j=1}^m (\int_{\R^{d_j}}f_j)^{p_j}
\end{equation}
holds for some $C<\infty$ if and only if the following hold:
\begin{align}\label{bl1}
&\sum_{j=1}^m p_j d_j =d, \\
\label{bl2}
dimV&\leq \sum_{j=1}^m p_j dim(L_j V)\text{ for any subspace }V\subseteq \R^d.
\end{align}
\end{theorem}

We denote $(\mathbf{L}, \mathbf{p})$ by the collection of $\lbrace L_j \rbrace_{1\leq j \leq m}$ and $\lbrace p_j \rbrace_{1\leq j \leq m}$. Also, we denote $\BL(\mathbf{L}, \mathbf{p})$ by the smallest constant $C$ for which  \eqref{BL} holds for all input data $f_1, \dots, f_m$.

To prove theorem \ref{MLR}, we need the following generalization of multilinear Kakeya estimate, which can be viewed as perturbation of Brascamp-Lieb inequality.

\begin{theorem}[\cite{BBFL, Z}]
\label{MLK}
Suppose $(\mathbf{L}, \mathbf{p})$ is a Brascamp-Lieb datum for which $\BL (\mathbf{L},\mathbf{p})<\infty$, $L_j : \R^d \rightarrow \R^{d_j}$ for each $j$. Then there exists $\nu >0$ such that, for every $\epsilon>0$,
\begin{equation}
\label{multi-kakeya}
\int_{\left[-1,1\right]^d}\prod_{j=1}^m (\sum_{T_j \in \mathbb{T}_j}\chi_{T_j})^{p_j}\leq C\delta^{d}\prod_{j=1}^m (\#\mathbb{T}_j)^{p_j}
\end{equation}
holds for all finite collections $\mathbb{T}_j$ of $\delta$-neighborhoods of $(d-d_{j})$-dimensional affine subspaces of $\mathbb{R}^d$ which, modulo translation, are within a distance $\nu$ of the fixed subspace $V_j \coloneqq ker L_j$.
\end{theorem}

The estimate  \eqref{multi-kakeya} was proved with the bound $C\delta^{d-\epsilon}$ in \cite{BBFL} 
and the $\delta^{-\epsilon}$ loss in the bound was removed later in \cite{Z}.

Let $\Sigma_j : \overline{U}_j\rightarrow \R^{d}$ be smooth parametrizations of $d_j$-dimensional submanifold $S_j$, where $U_j$ is a bounded open set in $\R^{d_j}$. Then, the associated extension operator is defined by 
$$E_j g(\xi)=\int_{U_j}e^{i\xi \cdot \Sigma_j (u)}g(u)du, \qquad \xi\in \R^{d}.$$
The following is an easy consequence of \cite[Theorem 1.3]{BBFL}. 

\begin{theorem}\label{MLR}
Let $\mathbf a=\{a_j\}_{1\le j\le m}$, $a_j\in \mathbb R^{d_j}$ and 
$\mathbf L(\mathbf a)=\{ d\Sigma_{j} (a_j))^*\}_{1\le j\le m}$. 
Suppose that $\BL (\mathbf{L}(\mathbf a),\mathbf{p})\le C_0$ for all $\mathbf a\in \overline{U}_1\times\dots\times \overline{U}_m$ and some constant $C_0$. Then, for every $\epsilon>0$, there exists a constance 
$C=C(\epsilon)$ such that 
\begin{equation}
\label{local}
\int_{B(0,R)}\prod_{j=1}^m \vert E_j g_j \vert ^{2p_j}\leq C R^{\epsilon}\prod_{j=1}^m \Vert g_j \Vert_{L^2 (U_j )}^{2p_j}
\end{equation}
holds for all $g_j \in L^2(U_j)$, $1\leq j \leq m$, and all $R\geq 1$.
\end{theorem}

We only  consider complex hypersurfaces which are given by $\Sigma_j(z)=(z,\phi(z))$, $z\in U_j$, where $\phi$ is a holomorphic function. In other words, we consider the case  $d_j =2n-2$, $d=2n$, and regard $U_j$ as a subset of $\C^{n-1}$ rather than $\R^{2n-2}$, and  $\C^n$ replaces $\R^{2n}$ via the obvious identification. 
It is plausible to expect  that $R^{\epsilon}$ at the right hand side of \eqref{local} is removable. But this is known only  for some special cases and the problem is left open in most of cases. $R^{\epsilon}$ can be replaced by $(logR)^{\kappa}$ for a suitable constant $\kappa$, see \cite{Z}.  

\begin{lemma}\label{decay_est} Let $\phi:\mathbb C^{n-1}\to \mathbb C$ be a holomorphic function on the support of $\chi\in C_c^\infty(\mathbb C^{n-1})$. Suppose $\det H\phi\neq 0$ on the support of $\chi$, then 
\[ \Big| \int e^{i w\odot (z,\phi(z))} \chi(z) dz\Big|\lesssim |w_n|^{n-1}.\]
\end{lemma}

\begin{proof}
Let us write $w_n=s+it$, and $z=x+iy$, $x,y\in \mathbb R^{n-1}$.   
For simplicity let us set \[(\rp(x,y), \ip(x,y)):=(\Re\phi (x+iy), \Im \phi(x+iy)).\]  By the stationary phase method, it is sufficient to show that the determinant of 
the hessian matrix of $s\rp (x,y)+t\ip(x,y)$ is comparable to $|(s,t)|^{2n-2}$. That is to say, 
\begin{equation}
\label{det}  
\bigg|\det \begin{pmatrix}   s\rp_{xx}''+t\ip_{xx}'' &  s\rp_{xy}''+t\ip_{xy}'' \\
s\rp_{xy}''+t\ip_{xy}  &   s\rp_{yy}''+t\ip_{yy}''
\end{pmatrix}\bigg|=(s^2+t^2)^{(n-1)} |\det H\phi(z)|^2.
\end{equation}
Here $\rp_{xx}''$ denotes the matrix $(\partial_{x_i}\partial_{x_j}\rp)$, and  similarly $\rp_{xy}''$, $\rp_{yy}''$ also denote the matrices  $(\partial_{x_i}\partial_{y_j}\rp)$, $(\partial_{y_i}\partial_{y_j}\rp)$, respectively.
In fact, we may write the phase function 
$w\odot (z,\phi(z))=|w_n|  \frac{w}{|w_n|}\odot (z,\phi(z))$. Then \eqref{det} shows 
the determinant of the hessian matrix of $\frac{w}{|w_n|}\odot (z,\phi(z))$ as a function of $x,y$ is bounded away from zero. 
Thus the standard  stationary phase method gives the desired estimate.

To see \eqref{det} note that  $\nabla_z\phi=\rp_x'+i\ip_x'$, and    
$     \det H\phi=  \det(\rp_{xx}''+i\ip_{xx}'') $  since $\phi$ is holomorphic. 
  Meanwhile,  $
(s\mathrm{I}-it\mathrm {I})H\phi=(s\rp_{xx}'' +t\ip_{xx}'')+i(s\ip_{xx}''-t\rp_{xx}'').
$
Hence, using \eqref{elementary}, 
\[   
\Big|\det((s\mathrm {I}-it\mathrm{I})H\phi\Big|^2=
\det \begin{pmatrix}   s\rp_{xx}''+t\ip_{xx}'' &  s\ip_{xx}''-t\rp_{xx}''\\
t\rp_{xx}'' -s\ip_{xx}'' &   s\rp_{xx}'' +t\ip_{xx}'' 
\end{pmatrix}.\] 
By the Cauchy-Riemann equation it follows that  $\ip_{xx}''=-\ip_{yy}'',$ $\rp_{xx}''=-\rp_{yy}''$, $\ip_{xx}''=-\rp_{xy}''$, and $\rp_{xx}''=\ip_{xy}''.$\footnote{$\rp_x'=\ip_y'$, $\rp_y'=-\ip_x'$.}   Thus, 
the right hand side is equal to
\[  \det \begin{pmatrix}   s\rp_{xx}''+t\ip_{xx}'' &  -s\rp_{xy}''- t\ip_{xy}''\\
s\rp_{xy}''+t\ip_{xy}'' &   -s\rp_{yy}''-t\ip_{yy}''
\end{pmatrix}.\] 
Therefore \eqref{det} follows. 
\end{proof}

\section{Multilinear restriction estimates for complex hypersurfaces}
Even though we have quite  general mulitilinear restriction estimates in Theorem \ref{MLR}, applying those estimates to particular cases is another matter. 
We need to reformulate those estimates in favorable forms which  suit for deducing linear restriction estimates for the complex surfaces.  
This is the place where the assumption that the function $\phi$ is holomorphic plays a role. The assumption significantly simplifies the description of the conditions which 
guarantee mutilinear estimates. 
However, this is not enough for our purpose since we also need general $k$-linear estimates with $k<n$.  
These estimate can not be directly deduced from the $n$-linear estimates. In particular the condition \eqref{bl1} is not generally satisfied by 
these k-linear estimates. Nevertheless, the difficulty can be easily overcome by simple projection argument  and the induction argument due to  Guth \cite{G2}. 

\subsection{$n$-linear restriction estimate}
For a point $a\in \mathbb C^{n-1}$ we set 
$$
n(\phi,a)=\left(\overline{\partial_{1} \phi}(a),\overline{\partial_{2} \phi}(a),\cdots,\overline{\partial_{{n-1}} \phi}(a),-1\right)\in \mathbb C^n,
$$
where $\partial_j=\partial_{z_j}$ is the complex derivative. This is the normal vector to the surface at $(a_1, a_2,\cdots ,a_{n-1}, \phi(a))$ with respect to the usual Hermitian inner product on $\C^n$. We also set 
 $$\mathbf n^\phi(a)=\frac{n(\phi,a)}{|n(\phi,a))|}.$$  
We will see that the complex line(real plane) generated by $n^{\phi}$ is normal to the graph of $\phi$ which has codimension 2.
 The following is a  consequence of Theorem \ref{MLR}.

\begin{theorem}\label{mlrcn}
Let $U\subset \mathbb C^{n-1}$ be a bounded open set and  $U_j \subset U$, $ 1\leq j \leq n$. Suppose $\phi: \overline U \rightarrow \C$ is a holomorphic function and 
\begin{equation}
\label{normal}
\left|\det\begin{pmatrix}
\nphi(a_1),  \nphi (a_2), \cdots, 
\nphi(a_n)
\end{pmatrix}\right|>c,   \quad \forall a_i\in U_i,  \, i=1, \dots, n
\end{equation}
for some $c>0$. 
Then, $\forall \epsilon>0, \forall R\geq 1$, there is  a constant $C=C(\epsilon)$ such that 
\[\Big\|\prod_{j=1}^n  \mathcal E^\phi_{U_j} g_j \Big\|_{L^{\frac{2}{n-1}}(B(0,R))}\leq C_{\epsilon}R^{\epsilon}\prod_{j=1}^n \Vert g_j \Vert_{L^2 (U_j )}.\]
\end{theorem}

\begin{proof}  We rephrase the condition in the real valued from. We write $z_j=x_j+iy_j$, 
\[\phi=\phi_1 + i\phi_2,\]
and set 
 $$\Sigma(x_1,y_1,x_2,\cdots,y_{n-1})=(x_1,\cdots,y_{n-1},\phi_1(x_1,\cdots,y_{n-1}),\phi_2(x_1,\cdots,y_{n-1})).$$ 
Under this identification  
$L_j = (d(\Sigma (a_j)))^*$ is given by 
$$ L_j=\left(\begin{array}{c|cc}
 & \frac{\partial\phi_1}{\partial x_1}(a_j) & \frac{\partial\phi_2}{\partial x_1}(a_j)\\
I_{2n-2} & \vdots  & \vdots\\
 & \frac{\partial \phi_1}{\partial y_{n-1}}(a_j) & \frac{\partial \phi_2}{\partial y_{n-1}}(a_j)
\end{array}\right).$$
Then, by Theorem \ref{MLR} it suffices to show that $\BL(\mathbf{L},\mathbf{p})<\infty$ for $\mathbf{L}=\lbrace L_j \rbrace_{1\leq j \leq n}$, $\mathbf{p}=\lbrace \frac{1}{n-1} \rbrace_{1\leq j \leq n}$.  That is to say,  $(\mathbf{L}, \mathbf{p})$ verifies \eqref{bl1} and \eqref{bl2}.   
The condition \eqref{bl1} is clearly satisfied. For the condition \ref{bl2}, we need to show 
\begin{equation}
\label{dimcon}
dimV\leq \sum_{j=1}^n \frac{1}{n-1} \dim(L_j V)\text{ for  any subspace }V\subseteq \R^{2n}.
\end{equation}
This follows from 
\begin{lemma}
\label{haha}
$\sum_{j=1}^n \dim(L_j V)=\sum_{j=1}^n (2n-\dim(ker(L_j \vline_V)))\geq 2n^2-\dim V$.
\end{lemma}

Indeed, by this lemma  \eqref{dimcon} is equivalent to $\dim V\le 2n$, which is trivially true.   
\end{proof}

\begin{proof}[Proof of Lemma \ref{haha}]
The first equality is obvious by the dimension theorem. For the second, it is enough to verify $$\dim V\geq \sum_{j=1}^n \dim(ker(L_j \vert_V))=\sum_{j=1}^n \dim(ker(L_j)\cap V).$$ 
The kernel space of $L_j$ is generated by two vectors 
\begin{align*}
v_{1,j}=\Big(\frac{\partial\phi_1}{\partial x_1}(a_j),\frac{\partial\phi_1}{\partial y_1}(a_j),\cdots,\frac{\partial\phi_1}{\partial x_{n-1}}(a_j), \frac{\partial\phi_1}{\partial y_{n-1}}(a_j),-1,0\Big), \\
v_{2,j}=\Big(\frac{\partial\phi_2}{\partial x_1}(a_j),\frac{\partial\phi_2}{\partial y_1}(a_j),\cdots,\frac{\partial\phi_2}{\partial x_{n-1}}(a_j), \frac{\partial \phi_2}{\partial y_{n-1}}(a_j),0,-1\Big).
\end{align*} since these two vectors are orthogonal to all row vectors of $L_j$. Since $\phi$ is holomorphic, by the Cauchy-Riemann equation it follows that 
\[v_{2,j}=\Big(-\frac{\partial\phi_1}{\partial y_1}(a_j),\frac{\partial\phi_1}{\partial x_1}(a_j),\cdots,-\frac{\partial \phi_1}{\partial y_{n-1}}(a_j),\frac{\partial\phi_1}{\partial x_{n-1}}(a_j), 0,-1\Big).\] 
Now we observe that 
\begin{equation}
\label{vmatrix}
|\det \begin{pmatrix}
v_{1,1}\cdots v_{1,n} \  v_{2,1}\cdots  v_{2,n}
\end{pmatrix} |
=\left| \det \begin{pmatrix}
n(\phi,a_1) \ n(\phi,a_2)  \cdots  n(\phi,a_n)
\end{pmatrix}\right|^2.
\end{equation}
Here we regard the vectors as column vectors. 
   Indeed, if we denote by $B$ the $n\times n$ matrix with the $j$-th column $(\frac{\partial\phi_1}{\partial x_1}(a_j),\cdots,\frac{\partial\phi_1}{\partial x_{n-1}}(a_j),-1)$, $j=1,\dots, n$ and by $D$  the $n\times n$ matrix with the $j$-th column $(\frac{\partial\phi_1}{\partial y_1}(a_j),\cdots,\frac{\partial\phi_1}{\partial y_{n-1}}(a_j),0)$. Then after rearrangement we note that 
$$
 |\det \begin{pmatrix}
v_{1,1}\cdots v_{1,n} \  v_{2,1}\cdots  v_{2,n}
\end{pmatrix} |=\det  \begin{pmatrix}  B & D \\ -D & B\end{pmatrix}.
$$
Now recall the elementary identity 
\begin{equation}\label{elementary}
  \det  \begin{pmatrix}  B & D \\ -D & B\end{pmatrix}=|\det(B+iD)|^2
  \end{equation}
  which is valid for any square matrix $B$ and $D$.  Note that $B+iD$ is equal to the matrix $(\overline{n(\phi,a_1)},  \ \overline{n(\phi,a_2)}  \cdots  \overline{n(\phi,a_n)})$. Thus \eqref{vmatrix} follows.
   
From \eqref{vmatrix} and  the condition \eqref{normal} it follows that  $\lbrace v_{1,1},v_{2,1},\cdots ,v_{1,n},v_{2,n} \rbrace$ is a basis of $\R^{2n}$. Consequently, we have the desired $\dim V\geq \sum_{j=1}^n \dim(ker(L_j)\cap V).
 $  
\end{proof}

\subsection{$k$-linear restriction estimate with $k<n$}
The above theorem is an $n$-linear restriction estimate while $n$ is the complex dimension of the ambient space. 
Unfortunately, except the case $n=2$ this type of  multilinear restriction estimate alone is not sufficient to deduce linear estimate, and  we also need multilinear estimates with intermediate multilinearity. However, 
these estimates are not straightforward from Theorem \ref{mlrcn}. In fact,  since the multilinear estimates in \cite{BBFL} were obtained under assumption that the Bracamp-Lieb inequality is finite, the expected estimates are subject to the scaling condition \eqref{bl1}, which is not satisfied with $k<n$. 
To get around this, instead of deducing the desired estimate from the existing estimate we directly prove them by adopting the strategy \cite{BBFL}  which was used for the proof of Theorem \ref{mlrcn}.   For this purpose  we first need to show  suitable multilinear Kakeya estimates associated with the complex surfaces.

\begin{definition} 
For $v=(a_1+ib_1, \cdots , a_n+ib_n)\in \mathbb C^{2n}$, set 
$\mathrm I(v)=(a_1,b_1, \cdots , a_n, b_n)$ $\in \mathbb R^{2n}$. 
For $v_1, \ldots, v_k\in \mathbb C^{2n}$, we define 
$$ 
\vert v_1\wedge v_2\wedge \cdots \wedge v_k\vert 
:= \big|\det \big(\mathrm I(v_1),  \mathrm I(iv_1), \cdots,  \mathrm I(v_k),  \mathrm I(iv_k), w_1, \cdots ,w_{2n-2k} \big)\big|.
$$
where  $\lbrace w_1, w_2, \cdots, w_l \rbrace$ is an orthonormal basis of the orthonormal complement of 
the subspace $\mathrm{span} \{\mathrm I(v_1),  \mathrm I(iv_1), \cdots,  \mathrm I(v_k),  \mathrm I(iv_k)\}$.
\end{definition}

Note that the definition does not depend on particular choices of bases  $w_1, w_2, \cdots, w_l$. 
Clearly, the value $\vert v_1\wedge v_2\wedge \cdots \wedge v_k\vert$ quantifies  degree of transversality between subspaces 
$\sspan\{v_i\}$ provided $|v_i|\sim 1$.  
Using this notion of transversality, we obtain multilinear restriction estimate with multiplicity smaller than $n$.

\begin{theorem}\label{mlrk}
Let $2\le k \le n-1$ be an integer. 
Let $U\subset \mathbb C^{n-1}$ be a bounded open set and  $U_j \subset U$, $ 1\leq j \leq k$. Suppose $\phi: \overline U \rightarrow \C$ is a holomorphic function and 
\begin{equation}
\label{normal}
\big| \nphi(a_1)\wedge \nphi (a_2)\wedge \cdots \wedge \nphi(a_k) \big|>c,   \quad \forall a_i\in U_i,  \, i=1, \dots, k
\end{equation}
for some $c>0$. 
Then, $\forall \epsilon>0, \forall R\geq 1$, there is  a constant $C=C(\epsilon)$ such that 
$$
 \Big\Vert \prod_{j=1}^k\mathcal E^\phi_{U_j} g_j\Big\Vert_{L^{\frac{2}{k-1}}(B(0,R))} \leq C_{\epsilon}R^{\epsilon}\prod_{j=1}^k \Vert g_j \Vert_{L^2 (U_j )}. 
$$
\end{theorem}

To prove this, we need the following form of  multilinear Kakeya estimate which we prove by adapting  the argument in \cite{BBFL}.
Once  Theorem \ref{mul_kakeya} below is obtained, one can prove Theorem \ref{mlrk} routinely following the argument in \cite{BBFL} which deduces mutilinear restriction estimate from general multilinear Kakeya estimate.  So, we omit proof of  Theorem \ref{mlrk}.  

\begin{theorem}
\label{mul_kakeya} 
For $1\leq j \leq k$, let $\mathbb{U}_{j,\delta}$ be a collection of $\delta$-neighborhoods  $U_{i,j}$  of  1-dimensional affine $\C$-subspaces  $\sspan\{ v_{i,j}\}$ of $\C^n$. 
Suppose  $|v_{i,j}|\sim 1$ and there is a fixed constant $c>0$  such that 
\[ \vert v_{i_1,1}\wedge v_{i_2,2}\wedge\cdots\wedge v_{i_k,k}\vert\geq c. \]
Then, for every $\epsilon >0$,
\[ \int_{\left[-1,1\right]^{2n}}\prod_{j=1}^k (\sum_{U_{i,j} \in \mathbb{U}_{j,\delta}}\chi_{U_{i,j}})^{\frac{1}{k-1}}\leq C_{\epsilon}\delta^{2n-\epsilon}\prod_{j=1}^k (\#\mathbb{U}_{j,\delta})^{\frac{1}{k-1}}. \]
\end{theorem}

It is likely that $\delta^{-\epsilon}$ can be removed but the current estimate  is good enough for our purpose.  
A similar estimate of lower level of multilinearity   was obtained in \cite{BCT} (see Theorem 5.1) for the typical  mutlilinear Kakeya case.  It was shown by monotonicity of heat flow. However the following argument  is quite flexible, so it can be used to deduce estimate of lower level multilinearity  from various scaling invariant  multilinear estimates. 

By decomposing the collection  $\mathbb{U}_{j,\delta}$ along the directions and the stability of the Brascamp-Lieb constant (see Lemma \ref{test}, and \cite{BBFL}), in order to show 
Theorem \ref{mul_kakeya} it suffices to prove the following reduced version.

\begin{prop}\label{mlkck}
Let $L_j$ be linear maps from $\C^n$ to $\C^{n-1}$ whose kernels are 1-dimensional $\C$-subspaces of $\C^n$ spanned by $v_j$ satisfying  $|v_{j}|= 1$,  for $1\leq j\leq k$. Suppose 
\begin{equation}
\label{k_trans}
\vert v_1\wedge v_2\wedge \cdots \wedge v_k\vert >c
\end{equation}
for some $c>0$. Then there exists $\nu >0$ such that, for every $\epsilon>0$,
$$\int_{\left[-1,1\right]^{2n}}\prod_{j=1}^k (\sum_{U_{i,j} \in \mathbb{U}_{j,\delta}}\chi_{U_{i,j}})^{\frac{1}{k-1}}\leq C_{\epsilon}\delta^{2n-\epsilon}\prod_{j=1}^k (\#\mathbb{U}_{j,\delta})^{\frac{1}{k-1}}$$
holds for all finite collections $\mathbb{U}_{j,\delta}$ of $\delta$-neighborhoods of 1-dimensional affine $\C$-subspaces  $\sspan\{ v_{i,j}\}$ of $\mathbb{C}^n$ provided that  the direction of $U_{i,j}$, $v_{i,j}$\footnote{Whenever we mention the direction vector, it is assumed to have unit length.} 
 is contained in the $\nu$-neighborhood of $V_j \coloneqq ker L_j$.
\end{prop}

The significance of this form is that it no longer needs to satisfy the dimension condition \eqref{bl1} which is necessary for the Brascamp-Lieb inequality. We first consider the case where $v_{i,j}$ is contained in the $ker L_j$. 

\begin{lemma}
\label{test} Suppose $\nu=0$ in Theorem \ref{mlkck}, that is to say, all $v_{i,j}$ are contained in the $ker L_j$. Then, the following inequality holds.
\begin{equation}
\label{simple}
\int\prod_{j=1}^k (\sum_{U_{i,j} \in \mathbb{U}_{j,\delta}}\chi_{U_{i,j}})^{\frac{1}{k-1}}\leq C\delta^{2n}\prod_{j=1}^k (\#\mathbb{U}_{j,\delta})^{\frac{1}{k-1}}.
\end{equation}
The constant $C$ remains uniformly bounded under small perturbation  $v_1,\dots, v_k$. 
\end{lemma}

\begin{proof}
We consider the integral over $\mathbb C^n$ as a double integral over the product spaces of $V=\sspan\lbrace v_1, v_2, \cdots, v_k \rbrace$ and its orthonormal complement.   After suitable change of coordinates, we may assume that  $V\times V^\perp=\C^k\times\C^{n-k}$. 
We write the left hand side of \eqref{simple} as follows:
\[\mathcal I:=\int_{\C^{n-k}}\int_{\C^k}\Big(\prod_{j=1}^k \sum_{U_{i,j}\in\mathbb{U}_{j,\delta}}\chi_{U_{i,j}}(x,y)\Big)^{\frac{1}{k-1}}dxdy.\]
We  may write 
\begin{equation}
\label{product}
\sum_{U_{i,j}\in \mathbb{U}_{j,\delta}}\chi_{U_{i,j}}=\sum_{U_{i,j}\in  \mathbb{U}_{j,\delta}} \chi_{B_{i,j,\delta}}\circ L_j
\end{equation}
for some  $B_{i,j,\delta}\subset \mathbb C^{n-1}$ which are balls of radius $\sim \delta$.

Let us set $W_j=\sspan\{L_j(v_1), \cdots,  L_j(v_k) \}$ and  consider the map $\widetilde {L}_j:\mathbb C^k\to W$ which is given by 
$\widetilde {L}_j(x)=L_j(x,0)$.  Then,  from \eqref{k_trans}  it is easy to see that 
$L_1,\cdots, L_k$  satisfy \eqref{bl1} and \eqref{bl2} with $m=k$ and $p_j=1/(k-1)$. In fact we have already checked this in the proof of Theorem \ref{mlrcn}.
Now taking 
\[ f_j^y(u)=\sum_{U_{i,j}\in  \mathbb{U}_{j,\delta}} \chi_{B_{i,j,\delta}}(u+ L_j(0,y)),  \  j=1,\dots, k,\] 
by Theorem \ref{BL}  we have
\begin{align}
\label{lalala}
\int_{\C^k}\prod_{j=1}^k (f^y_{j})^{\frac{1}{k-1}}(\widetilde L_j(x))dx
    \lesssim \prod_{j=1}^k \Big(\int_{W_j} f_{j}^y(u) du\Big)^{\frac{1}{k-1}}.
    \end{align}
Combining this with \eqref{product} yields    
    \begin{align*}
   \mathcal I \lesssim  \int   \prod_{j=1}^k\Big( \int_{W_j} \sum_{U_{i,j}\in  \mathbb{U}_{j,\delta}} \chi_{B_{i,j,\delta}}(u+ L_j(0,y))du\Big)^{\frac{1}{k-1}}  dy.
\end{align*}
Using H\"older's inequality, 
  \begin{align*}
   \mathcal I \lesssim    \prod_{j=1}^k  \Big( \int \Big( \int_{W_j} \sum_{U_{i,j}\in  \mathbb{U}_{j,\delta}} \chi_{B_{i,j,\delta}}(u+ L_j(0,y))du\Big)^{\frac{k}{k-1}}  dy\Big)^\frac1k.
\end{align*}
Thus, to obtain \eqref{simple} it is sufficient to show that 
\[ \mathcal {I\!I}_j:=
 \int \Big( \int_{W_j} \sum_{U_{i,j}\in  \mathbb{U}_{j,\delta}} \chi_{B_{i,j,\delta}}(u+ L_j(0,y))du\Big)^{\frac{k}{k-1}}  dy
 \lesssim \delta^{2n} (\#\mathbb{U}_{j,\delta})^{\frac{k}{k-1}}.
   \]
By decomposing 
 $\C^{n-k}$ into boundedly overlapping balls $B$ of  side length  $\delta$, we have 
\begin{align*}
\mathcal {I\!I}_j  &\lesssim   \sum_B\int_B \delta^{2k} \Big(\#\lbrace U_{i,j} : U_{i,j}\cap (\C^k\times B) \neq \emptyset \rbrace\Big)^{\frac{k}{k-1}}dy \\
 & \sim   \delta^{2n}\sum_B   (\#\lbrace U_{i,j} : U_{i,j}\cap (\C^k\times B) \neq \emptyset \rbrace)^{\frac{k}{k-1}}
 \\
 & \sim   \delta^{2n} \Big(\sum_B   (\#\lbrace U_{i,j} : U_{i,j}\cap (\C^k\times B) \neq \emptyset \rbrace) \Big)^{\frac{k}{k-1}}.
\end{align*}
Clearly, $\sum_B  \#\lbrace U_{i,j} : U_{i,j}\cap (\C^k\times B) \neq \emptyset \rbrace\lesssim  \#\mathbb{U}_{j,\delta}$. 
Thus, we get the desired inequality.

Finally the last statement is consequence of the stability of Brascamp-Lieb constant \cite{BBFL}, which gives 
\eqref{lalala} with a uniform constant $C$ under small perturbation of the kernel of $v_1, \dots, v_k$.  Thus, 
from the argument above we see that \eqref{simple} holds with a uniform $C$. 
\end{proof}

We prove Theorem \ref{mlkck} by perturbing the directions in the above lemma. 
The  idea is basically  due to Guth \cite{G2}.

\begin{proof}[Proof of Theorem \ref{mlkck}]
We continue to denote  the direction of $U_{i,j}$ by $v_{i,j}\in \mathbb C^{n}$ as in the above, and we  may also assume $|v_{i,j}|=1$. 
Let $B(\delta, \nu)$ be the smallest bound $C$ such that 
$$\int_{\left[-1,1\right]^{2n}}\prod_{j=1}^k (\sum_{U_{i,j} \in \mathbb{U}_{j,\delta}}\chi_{U_{i,j}})^{\frac{1}{k-1}}\leq C\delta^{2n}\prod_{j=1}^k (\#\mathbb{U}_{j,\delta})^{\frac{1}{k-1}}$$ holds. To complete proof, 
it is sufficient  to show that $B(\delta, \nu)\leq C_{\epsilon}\delta^{-\epsilon}$.
This can be shown by the following iterative inequality for $B(\delta, \nu)$. 

\begin{lemma}
\label{induction}
There exists a number $k$ independent of $\delta$ and $\nu$, such that
\begin{equation}
\label{relation}
B(\delta, \nu)\leq k B(\frac{\delta}{\nu},\nu).
\end{equation}
\end{lemma}

Applying the inequality $l$ times, we have $B(\delta,\nu)\leq k^l B(\frac{\delta}{\nu^l},\nu)$. We only need to choose $\nu$ such that $\epsilon log \frac{1}{\nu}=log k$, then choose $l$ such that $\frac{\delta}{\nu^l}\sim 1$. So, $k^l\sim \delta^{-\epsilon}$ and the desired bound follows. 
\end{proof}

\begin{proof}[Proof of Lemma \ref{induction}]
We partition $\left[ -1,1 \right]^{2n}=\bigcup Q$ by cubes of side length $\sim \frac{\delta}{\nu}$. Then, we write 
\begin{equation}
\label{lala}
\int_{\left[-1,1\right]^{2n}}\prod_{j=1}^k (\sum_{U_{i,j} \in \mathbb{U}_{j,\delta}}\chi_{U_{i,j}})^{\frac{1}{k-1}}=\sum_{Q}\int_{Q}\prod_{j=1}^k (\sum_{U_{i,j} \in \mathbb{U}_{j,\delta,Q}}\chi_{U_{i,j}\bigcap Q})^{\frac{1}{k-1}},
\end{equation}
where $\mathbb{U}_{\j,\delta,Q}=\lbrace U_{i,j}\in\mathbb{U}_{j,\delta} : U_{i,j}\bigcap Q \rbrace\neq \phi$.

We focus on a single $Q$. If $U_{i,j}\in \mathbb{U}_{j,\delta,Q}$, we note that distance between $v_{i,j}$ and $kerL_j$ is less than $\nu$. So, since $Q$ has side length $\sim \frac{\delta}{\nu}$, there exists a $O(\delta)-$neighborhood $U_{i,j}'$ such that $U_{i,j}\bigcap Q\subset U_{i,j}'\bigcap Q$ and $U_{i,j}'$ is parallel to $kerL_j$. Now we can use Lemma \ref{test} to get
$$ \int_Q\prod_j \Big( \sum_{U_{i,j}\in\mathbb{U}_{j,\delta,Q}}\chi_{U_{i,j}'} \Big)^{\frac{1}{k-1}}\lesssim \delta^{2n} \prod_j \left( \#\mathbb{U}_{j,\delta,Q} \right)^{\frac{1}{k-1}}.$$
If $U_{i,j}\in \mathbb{U}_{j,\delta,Q}$, let $\tilde{U}_{i,j}=U_{i,j}+O(\frac{\delta}{\nu})$ such that $\tilde{U}_{i,j}$ contains $Q$. So,
\[  \prod_j \left( \#\mathbb{U}_{j,\delta,Q} \right)^{\frac{1}{k-1}}  \lesssim   \frac{1}{|Q|}\int_Q \prod_j \Big( \sum_{U_{i,j}\in\mathbb{U}_{j,\delta,Q}}\chi_{\tilde{U}_{i,j}}(x) \Big)^{\frac{1}{k-1}}dx.\]  
Combining these two estimates gives 
$$
 \int_Q\prod_j \Big( \sum_{U_{i,j}\in\mathbb{U}_{j,\delta,Q}}\chi_{U_{i,j}'} \Big)^{\frac{1}{k-1}}\lesssim \delta^{2n}\frac{1}{|Q|}\int_Q \prod_j \Big( \sum_{U_{i,j}\in\mathbb{U}_{j,\delta,Q}}\chi_{\tilde{U}_{i,j}}(x) \Big)^{\frac{1}{k-1}}dx.$$
Recalling \eqref{lala}, we put  all the estimates over  $Q$ together. Since $|Q|\sim (\delta/\nu)^{2n}$,  recalling the definition of 
$B(\delta, \nu)$, we get 
\begin{align*}
\int_{\left[-1,1\right]^{2n}}\prod_{j=1}^k (\sum_{U_{i,j} \in \mathbb{U}_{j,\delta}}\chi_{U_{i,j}})^{\frac{1}{k-1}}  
&\lesssim
 \nu^{2n}\int_{\left[-1,1\right]^{2n}}\prod_j \Big(\sum_{U_{i,j}\in\mathbb{U}_j}\chi_{\tilde{U}_{i,j}} \Big)^{\frac{1}{k-1}}\\
& \leq \nu^{2n}B(\frac{\delta}{\nu},\nu)\left(\frac{\delta}{\nu}\right)^{2n}\prod_j (\#\mathbb{U}_{j,\delta})^{\frac{1}{k-1}}.
\end{align*}
This gives the desired \eqref{relation}. 
\end{proof}


\begin{remark}\label{e-removal}
As mentioned before, it is an open problem whether it is possible to remove $R^{\epsilon}$ in the above estimate in Theorem \ref{mlrk}. But,  if we consider the estimate for $p>\frac{2k}{k-1}$,  we may remove  $R^{\epsilon}$ if $H\phi\neq 0$ on  $ \overline U$. This can be shown by the epsilon removal argument, see \cite[Appendix]{BG}.
\end{remark}

\section{Restriction estimates for complex hypersurfaces}

\subsection{Complex hypersurfaces given by quadratic polynomials}
Now we prove the linear restriction estimate. We first show Theorem \ref{main} with a quadratic function  $\phi$.
Once it is done, extension to general holomorphic function requires only small modification of the argument. 

\begin{prop}\label{m1}  Let $n\ge 3$ be an odd number and  
$\phi(z)=z^t M z$ for a nonsingular  symmetric  matrix $M$. Suppose $D$ is a bounded domain.
Then,  for $p>\frac{2(n+2)}{n}$,   there is a constant $C_p$ such that 
$
\Vert  \mathcal E^\phi_D f \Vert_{p}\leq C_p \Vert f \Vert_{p}.
$
\end{prop}

Let  $Q(a,r)$ be the cube centered at $a$ with side length $r$. 
For simplicity we also set 
\[\mathcal E=  \mathcal E^{z^t\! M z}_{Q(0,1)}. \] 
Let $R\ge 1$ and,  for given $p$,   we  define 
\[   \mathcal A_{p}(R) =\sup\{ \| \mathcal E f\|_{L^p(Q(0,R))} :  \|f\|_p\le 1\} .   \] 
To prove Proposition \eqref{m1}, by finite decomposition, translation and scaling,  it is enough to show   $\mathcal A_{p}(R) \le C$  for $p>\frac{2(n+2)}{n}$.
For this we make use of the following lemma. 

\begin{lemma}[Parabolic rescaling]
\label{pscaling}
Suppose $f$ is supported in  $Q(a,r)\in Q(0,1)$ and $r<  (\sqrt d(1+2\|M\|))^{-1}$. Then, 
$$\Vert \mathcal E f \Vert_{L^p_{Q(0,R)}}\leq Cr^{2(n-1)-\frac{4n}{p}}\mathcal A_{p}(R)\Vert f_{\alpha}\Vert_p.$$
\end{lemma}
\begin{proof}
We first note that
$$\mathcal E f(w)=\int_{Q(a,r)} e^{iw\odot (z,z^tMz)}f(z)dz.$$
Let $\psi(z)=z^tMz$. Then $\psi(z-a)=\psi(z)-\psi(a)-\nabla \psi(a)\cdot(z-a)$.
We also set $f_{a,r}(z)= r^{2n-2} f(a+rz)$.
Using these notations, we may write 
\begin{eqnarray*}
\big|\mathcal E f(w)\big| &=& \Big|\int_{Q(a,r)}e^{i\left((w'+w_n\nabla\psi(a))\odot z +w_n\odot \psi(z-a)\right)}f(z)dz\Big|
\\
&=&\Big|\int_{Q(0,1)}e^{i\left( r(w'+w_n\nabla\psi(a))\odot z +r^2  w_n \odot  \psi(z)  \right)}f_{a,r}(z)dz\Big|
\\
&=&  \big|\mathcal E f_{a,r}\big(r(w'+w_n\nabla\psi(a)), r^2  w_n \big)\big|.
\end{eqnarray*}

We integrate with  respect to $w$. By making change of variables $r(w'+w_n\nabla\psi(a))\to w'$ and $r^2 w_n\to w_n$
\begin{align*}
\Vert \mathcal E f\Vert^p_{L^p_{Q(0,R)}}
&\le r^{-(2n+2)} \int_{\vert w'\vert<\sqrt d(1+2\|M\|) rR,\vert w_n\vert<r^2R}\big| \mathcal E f_{a,r}(w) \big|^pdw\\
& \le r^{-(2n+2)} \int_{Q(0,R)}\big| \mathcal E f_{a,r}(w) \big|^pdw.
\end{align*}
Now, using  the definition of $\mathcal A_{p}$, we see 
\begin{align*}
\Vert \mathcal E f\Vert_{L^p_{Q(0,R)}} & \leq \mathcal A_{p}(R) r^{-\frac{2n+2}p} \|f_{a,r}\|_p.\end{align*}
This gives the desired bound by rescaling.
\end{proof}

\begin{proof}[Proof of Proposition \ref{m1}] Let $R>0$ be a large number and  fix a large number $1\ll K\ll R$.  Let $\{\mathfrak  q\}$ be a collection of essentially disjoint cubes with side length  $\sim K^{-1}$ which partitions 
 $Q(0,1)$ and let $\{\mathbf  Q\}$ be a collection of essentially disjoint cubes with side length  $\sim K$ which partitions 
 $Q(0,R)$. Thus
 \[ Q(0,1)=\bigcup \mathfrak q, \qquad   Q(0,R)=\bigcup \mathbf Q.\]  
We set 
\[ f_\mathfrak q=f\chi_\mathfrak q.\] 
So, we have $f=\sum_\mfq f_\mfq$.  
We will denote by $C(K)$ some powers of $K$ and this may vary from line to line.

We first   consider $\{\ce f_\mfq\}$ on each cube $\mbq$. 
An important observation is that,  {on} each cube $\mbq$,  $\ce f_\mfq$ behaves as if 
it were a constant. 
To make it precise we need a bit of manipulation.
Let $\eta\in\mathcal S(\mathbb C^n)$ such that $\widehat{\eta}(w)=1$ if $|w|\leq 1$ and $\widehat{\eta}(w)=0$ if $|w|\geq 2$. 
For ${\mfq}$ let $z_{\mfq}$ be the center of $\mfq$ and set 
\[\eta_{\mfq}(w)=\frac{e^{2\pi iw\cdot (z_{\mfq}, \psi(z_\mfq))}}{K^{2n}}\eta (\frac{w}{K}).\] 
Thus,  we have 
\begin{equation}\label{conv}
\ce f_{\mfq}(w)=\ce f_{\mfq}\ast\eta_{\mfq} (w).
\end{equation}

Let use denote the center of $\mbq$ by $w_\mbq$. 
Put $\zeta(x)=\max\limits_{|x'|\leq \sqrt d}|\eta(x+x')|^{\frac{1}{n}}$ and set 
\[  C_{\mfq}^{\, \mbq}:=K^{-2n^2}\left(\int \vert \ce f_{\mfq}(w_\mbq-s)\vert^{\frac{1}{n}} \zeta(\frac{s}{K}) ds\right)^n.
\] 
The following is a slight modification of the argument in \cite{Te} (see,  p.1024).

\begin{lemma} \label{compare} Suppose $w\in \mbq$, then we have
\begin{equation}
\label{compare1}
 |\ce f_{\mfq}(w)|  \lesssim   \cqq \lesssim   K^{-2n^2}\int \vert \ce f_{\mfq}(w -s)\vert \widetilde \zeta(\frac{s}{K}) ds,
\end{equation}
where $\widetilde \zeta(x)=\max\limits_{|x'|\leq \sqrt d}|\zeta(x+x')|$. Furthermore, let  
$\mfq_1, \mfq_2, \cdots, \mfq_k\in \{\mfq\}$. If $p\ge 1$ and $w\in \mbq$,  we have 
\begin{equation}
\label{compare2}  \Big(   \prod_1^k  C_{\mfq_j}^{\,\mbq}  \Big)^{\frac{p}{k}}
\lesssim  
C(K)
 \prod_{j=1}^k \int  \vert \ce f_{\mfq_j}(w-s_j)\vert^{\frac{p}{k}} \widetilde \zeta(\frac{s_j}{K}) ds_j .
 \end{equation}

\end{lemma}

\begin{proof}
We first observe that 
\begin{eqnarray*}
\Vert \ce f_{\mfq}(w-\cdot)\eta_{\mfq}(\cdot)\Vert_{1}
                 & \leq & \Vert \ce f_{\mfq}(w-\cdot)\eta_{\mfq}(\cdot)\Vert_{\infty}^{\frac{1}{n'}}\int|\ce f_{\mfq}(w-s)\eta_{\mfq}(s)|^{\frac{1}{n}}ds
                 \\
                 & \lesssim & K^{\frac{-2n}{n'}}\Vert \ce f_{\mfq}(w-\cdot)\eta_{\mfq}(\cdot)\Vert_{1}^{\frac{1}{n'}}\int|\ce f_{\mfq}(w-s)\eta_{\mfq}(s)|^{\frac{1}{n}}ds.
\end{eqnarray*}
Since  the Fourier transform of $\ce f_{\mfq}(w-\cdot)\eta_{\mfq}(\cdot)$ is  supported in a ball of radius $\lesssim \frac1K$, 
the last inequality follows from Bernstein's inequality. This immediately yields
\[ \int |\ce f_{\mfq}(w-s)\eta_{\mfq}(s)|ds 
\lesssim K^{\frac{-2n^2}{n'}} \left( \int |\ce f_{\mfq}(w-s)\eta_{\mfq}(s)|^{\frac{1}{n}}ds\right)^n . \]
Thefore, by \eqref{conv} and the above, we get 
\begin{align*}
|\ce f_{\mfq}(w)| 
                 & \lesssim  K^{-2n^2} \left( \int \vert \ce f_{\mfq}(w-s)\vert^{\frac{1}{n}}\vert\eta(\frac{s}{K})\vert^{\frac{1}{n}}ds \right)^n \\
                 & =    K^{-2n^2} \left( \int \vert \ce f_{\mfq}(w_\mbq-s)\vert^{\frac{1}{n}}\big|\eta(\frac{s+w-w_\mbq}{K})\big|^{\frac{1}{n}} ds \right)^n.
                 \end{align*}
                 Since $|w-w_\mbq|\le  K\sqrt d$, it follows that $\big|\eta(\frac{s+w-w_\mbq}{K})\big|^{\frac{1}{n}}\le  \zeta(\frac s K)$. Thus we have 
the desired $  |\ce f_{\mfq}(w)| \lesssim    \cqq $.  
For \eqref{compare2},  
note that 
\begin{align*}
 \Big(   \prod_1^k  C_{\mfq_j}^{\,\mbq}  \Big)^{\frac{p}{k}}
 &= C(K)\bigg(  \prod_{j=1}^k \int  \vert \ce f_{\mfq_j}(w-s_j)\vert^{\frac{1}{n}} \zeta(\frac{s_j-w+w_\mbq}{K}) ds_j\bigg)^{\frac{pn}k}
 \\
 &\lesssim  C(K)
 \prod_{j=1}^k \int  \vert \ce f_{\mfq_j}(w-s_j)\vert^{\frac{p}{k}} \zeta(\frac{s_j-w+w_\mbq}{K}) ds_j.
  \end{align*}
 For the last inequality we use H\"older's inequality and, as before,  \eqref{compare2}
 follows since $|w-w_\mbq|\le  K\sqrt d$. The second inequality in \eqref{compare1}  can be shown by the same argument. 
This completes proof. 
\end{proof}

Fix $2\le k\le n$ and  $\mbq$. We set 
\[ \mathcal Q^\mbq_s= \big \{\mfq:   \vert \cqq\vert < K^{2-2n}\max_{\mfq}\vert \cqq \vert \big \}, \quad  \mathcal   Q^\mbq_l=\big \{\mfq:   \vert \cqq\vert\ge K^{2-2n}\max_{\mfq}\vert \cqq \vert  \big \}.\]
By Lemma \ref{compare}, for $w\in \mbq$, 
\begin{equation}
\label{sm}  \sum_{\mfq\in \mathcal Q^\mbq_s}  |\ce f_\mfq(w)|\le \max_{\mfq}  \cqq \lesssim  {\rm I} :=   K^{-2n^2}  \int   \max_{\mfq}\vert    \ce f_{\mfq}(w -s)\vert \widetilde \zeta(\frac{s}{K}) ds.
\end{equation}
We sort $\mathcal Q^\mbq_l$  into two cases. 
It is clear that there are only the following two cases:
\\
{\bf Case 1:}
There are  $\mfq_1, \mfq_2, \cdots, \mfq_k$ such that
\begin{equation}\label{transv}
\vert \mathbf n^\phi( z_{1})\wedge \mathbf n^\phi( z_{2}) \wedge \cdots \wedge \mathbf n^\phi( z_{k}) \vert > \frac{c}{K^{2k}}, \quad \forall z_i\in \mfq_i,  i=1,\dots, k.
\end{equation}
{\bf Case 2:} There is a $(k-1)$-dim $\C$-subspace $V_{k-1}=V_{k-1}(\mbq)$ such that
\begin{equation}
\text{dist}(\mathbf n^\phi(\mfq),V_{k-1})\lesssim K^{-1}, \quad \forall \mfq\in \mathcal Q^\mbq_l.
\end{equation}

Now let $\mfq\in \mathcal Q^\mbq_l$.  In {\bf Case 1}, it follows that $\cqq \le K^{2n-2}   \Big(   \prod_1^k  C_{\mfq_j}^{\,\mbq}  \Big)^{\frac{1}{k}}$. 
Thus  from \eqref{compare1} and \eqref{compare2}  we have, for $w\in \mbq$,
\[ 
|\ce f_\mfq(w)|^p \lesssim  C(K)
 \prod_{j=1}^k \int  \vert \ce f_{\mfq_j}(w-s_j)\vert^{\frac{p}{k}} \widetilde \zeta(\frac{s_j}{K}) ds_j  
 \]
for some $\mfq_1, \mfq_2, \cdots, \mfq_k$ satisfying \eqref{transv}.  There are as many as $O(K^{2n-2})$ $\mfq$.
We make the right hand side independent of $\mbq$ by considering all the possible choices of $\mfq_1, \mfq_2, \cdots, \mfq_k$ satisfying \eqref{transv}.  Indeed,  after a simple manipulation we have  
\begin{equation}
\label{mult}   
(\sum_{\mfq\in \mathcal Q^\mbq_l}  |\ce f_\mfq(w)|)^p \lesssim   {\rm I\!I}:= C(K)  \sum_{\substack{(\mfq_1, \mfq_2, \cdots, \mfq_k)\\ \text{satisfying } \eqref{transv}}} 
 \prod_{j=1}^k \int  \vert \ce f_{\mfq_j}(w-s_j)\vert^{\frac{p}{k}} \widetilde \zeta(\frac{s_j}{K}) ds_j .
 \end{equation}

Now we consider {\bf Case 2}. Let us set $V_{k-1}'=\lbrace z :   \mathbf n^\phi(z)\in V_{k-1} \rbrace$.  Since the hessian matrix of $\phi$ is non-singular, from the inverse function theorem we see that $V_{k-1}'$ is a manifold of dimension $(2k-4)$.   And observe that, for $\mfq\in \mathcal Q_l^\mbq$, 
\begin{equation}\label{confined}
  \mfq\subset \{ z:\dist(z, V_{k-1}')\le C_{d, M}    K^{-1}\}.
  \end{equation}
Clearly we have 
$$\vert \widetilde w-w\vert \gtrsim  K^{-1}  \Longleftrightarrow \vert \mathbf n^\phi(\widetilde w)- \mathbf n^\phi(w)\vert \gtrsim K^{-1}.$$
Since each $\mfq$ is a cube of size $K^{-1}$ and contained in  a $CK^{-1}$-neighborhood of the $(2k-4)$-manifold  $V_{k-1}'$.  It follows that 
\begin{equation} 
\label{number}
 \#  \mathcal Q_l^\mbq\lesssim  K^{2k-4}. \end{equation} 
Since  $\ce f_{\mfq}$ is bounded by $\cqq$ on  $\mbq$,   it is easy to see 
$$ 
\int_{\mbq}  \big\vert \sum_{\mfq\in \mathcal Q_l^\mbq} \ce f_\mfq \big\vert^p dw 
\lesssim 
\Big( \int_{\mbq}\big| \sum_{\mfq\in \mathcal Q_l^\mbq}\ce f_{\mfq} \big|^2 dw\Big)
                        \Big( \sum_{\mfq\in \mathcal Q_l^\mbq} \cqq \Big)^{p-2}.
$$
Let $\mu$ be a smooth function satisfying that  $\mu \sim1$ on $\mbq$, and $\widehat \mu$\, is supported in a ball of radius $\sim K^{-1}$. 
Since 
$ \int_{\mbq}|\sum_{\mfq\in \mathcal Q_1^\mbq}\ce f_{\mfq}|^2dw\lesssim \int |\sum_{\mfq\in \mathcal Q_1^\mbq}\ce f_{\mfq}|^2 |\mu|^2dw,$
Plancherel's theorem and orthogonality between $\ce f_{\mfq}\cdot \mu$ give
\begin{eqnarray*}
\int_{\mbq}  \big\vert \sum_{\mfq\in \mathcal Q_l^\mbq} \ce f_\mfq \big\vert^p dw 
& \lesssim & |\mbq| 
\Big( \sum_{\mfq\in \mathcal Q_l^\mbq} (\cqq)^2 \Big)\Big( \sum_{\mfq\in \mathcal Q_l^\mbq} \cqq \Big)^{p-2}
\\
& \lesssim  & (\#\mathcal Q_l^\mbq)^{1-\frac{2}{p}+(p-2)(1-\frac{1}{p})}\sum_{\mfq}(\cqq)^{p}  |\mbq| . 
\end{eqnarray*}
By \eqref{compare1} and H\"older's inequality,  $ (\cqq)^p\lesssim   K^{-2n} \chi_\mbq(w)\int \vert \ce f_{\mfq}(w -z)\vert^p\widetilde \zeta(\frac{z}{K}) dz$. Thus, combining the above with \eqref{number}, we have
\begin{eqnarray}
\int_{\mbq}  \big\vert \sum_{\mfq\in \mathcal Q_l^\mbq} \ce f_\mfq \big\vert^p dw
& \lesssim & 
  \int_{\mbq} {\rm I\!I\!I}\, dw , 
\label{confined}
\end{eqnarray}
where 
\[ {\rm I\!I\!I}=K^{(2k-4)(p-{2})}   \int   \sum_{\mfq} |\ce f_{\mfq}(w -s)|^p \widetilde \zeta(\frac{s}{K})  ds.\]

Therefore,  putting \eqref{sm}, \eqref{mult}, and \eqref{confined}  together, 
we get 
\[  \int_{\mbq} |\sum_\mfq \ce   f_\mfq (w)|^p  dw \lesssim \int_\mbq {\rm I}^p {\rm+ I\!I+ I\!I\!I}\, dw .  \]
Note that ${\rm I}$, ${\rm I\!I}$, and ${\rm I\!I\!I}$ are independent of $\mbq$. 
  Thus, by summation along $\mbq$  we have
 \[ \int_{Q(0,R)} | \ce   f(w)|^p  dw  \lesssim   \int_{Q(0,R)} {\rm I}^p+ {\rm I\!I+ I\!I\!I}\,  dw.  \] 
 
By \eqref{sm}, the imbedding $\ell^p\hookrightarrow\ell^\infty$, and Lemma \ref{pscaling} we have 
\begin{align*}
 \int_{Q(0,R)} {\rm I}^p\, dw 
                        &\lesssim  K^{-2n}  \int   \sum_{\mfq}  \int_{Q(0,R)} \vert    \ce f_{\mfq}(w -s)\vert^p \, dw \widetilde  \zeta(\frac{s}{K}) ds
                        \\
                        \lesssim  & \ K^{ {4n}-(2n-2)p}  \mathcal A_p^p(R)     \sum_{\mfq} \| f_{\mfq}\|_p^p
                        \  \le  \ K^{ {2n}-(2n-2)p} \mathcal A_p^p(R) \|f\|_p^p\,.\end{align*}
By Theorem \ref{mlrk} we have, for $p>\frac{2k}{k-1}$, 
\begin{align*}
 \int_{Q(0,R)} {\rm I\!I}\, dw 
                        &\lesssim  C(K)  \sum_{\substack{(\mfq_1, \mfq_2, \cdots, \mfq_k)\\ \text{satisfying } \eqref{transv}}}  \ 
\int   \int_{Q(0,R)}   \prod_{j=1}^k \vert \ce f_{\mfq_j}(w-s_j)\vert^{\frac{p}{k}}  dw\,  \prod_{j=1}^k  \widetilde \zeta(\frac{s_j}{K})   ds_j  \\
                        &\lesssim  C(K)  \|f\|_2^p\,.
                                                 \end{align*}
Repeating the same argument for ${\rm I}^p$, it is easy to see that 
\begin{align*}
\int  {\rm I\!I\!I}\,  dw \lesssim   K^{2(2(n-k+2)-p(n-k+1))} \mathcal A_p^p(R) \|f\|_p^p.
\end{align*}
 Combining all together we have
 \[ \Vert \mathcal E f\Vert^p_{L^p_{Q(0,R)}}\lesssim C(K) \|f\|_p + CK^{-\epsilon_\circ}\mathcal A_p(R) \|f\|_p\]
for some $\epsilon_\circ>0$ provided that 
\[p> \max_{2\le k\le  n-1}\Big(\frac{2(n-k+2)}{n-k+1},   \frac{2k}{k-1}\Big).\] 
The above inequality is valid for any $f$. With a sufficiently large $K$ such that  $CK^{-\epsilon_\circ}\le 1/2$,  we have 
$\mathcal A_p(R)\le C(K)$. We may choose $k=\frac{n}{2}+1$ since $n$ is even. 
This gives   $\mathcal A_p(R)\le C(K)$ for $p>\frac{2(n+2)}{n}$, and completes the proof.
\end{proof}

\subsection{General holomorphic function}
Now we consider the general holomorphic function $\phi$ with nonzero $\det H\phi$ and prove Theorem \ref{main}. The proof of Proposition \ref{m1} works with a slight modification. Since  $\phi$ is no longer a quadratic polynomial, we need to modify the induction quantity 
$\mathcal A_p(R)$. 

Let denote $\mathfrak F(\delta_\circ,r)$ by the collection of function $g$ which satisfies 
\vspace{-7pt}
\begin{enumerate}
[leftmargin=0.24in]
\item  $g$ is analytic on $\overline {Q(0,r)}$.
\item  $g(0)=0,$ and $\nabla g(0)=0$.
\item  $|\partial^{\alpha} _z (g(z)-\frac12 z\cdot z)|<\delta_\circ $  for all $z\in\overline {Q(0,r)}$ and $2\le |\alpha|\le  2n+2$.
\end{enumerate}

Considering the expansion  
$g(z+z_0)- g(z_0)-\nabla g(z_0)\cdot z= \frac12 z^tHg(z_0) z+O(|z|^3),$ we set 
\[g_{z_0}^{\epsilon}(z):=\frac{\ve^{-2}}2\big(g(\ve z+z_0)- g(z_0)-\nabla g(z_0)\cdot \ve z\big)= \frac12  z^tHg(z_0) z+O(\ve|z|^3).\]
Since the matrix $Hg(z_0)$ varies, it is desirable to normalize the second order term. 
The matrix $Hg(z_0)$ is symmetric but not hermitian. However, by 
Takagi's decomposition  
we may write  \[Hg(z_0)=U^t(z_0)D(z_0)U(z_0)\] with a diagonal matrix $D(z_0)$  and a unitary matrix $U(z_0)$ (for example, see \cite{HJ}). 
Let $\lambda_1(z_0), \dots, \lambda_{n-1}(z_0)$ be the diagonal entries of $D(z_0)$.
Since $Hg$ is non singular,  we clearly see there are $c, C>0$ such that $ c\le |\lambda_1(z_0)|, \dots, |\lambda_{n-1}(z_0)|\le C$. 
Let us define 
\[ [g]_{z_0}^{\epsilon}(z):=  g_{z_0}^{\epsilon}(  (\sqrt{D(z_0)}U(z_0))^{-1})z)=\frac12 z\cdot z+O(c^{-1}\ve|z|^3). \]
This shows that any analytic function with nonsingluar hessian matrix can be harmlessly transformed to a function contained in $\mathfrak F(\delta_\circ)$ by an affine transform. Thus, by decomposing the operator we may regard the extension operator $\ce_{D}^\phi$ as a finite sum of  $\ce_{Q(0,1)}^\phi$ with $\phi\in \mathfrak F(\delta_\circ,1+\epsilon_\circ)$ with sufficiently small $\delta_\circ, \epsilon_\circ>0$. Therefore,  for the proof  of  Theorem \ref{main} we only need to consider  the extension operator given by this type of normalized $\phi\in \mathfrak F(\delta_\circ,1+\epsilon_\circ)$.  
The above argument also shows 

\begin{lemma}\label{scaled} Let $\delta_\circ$, $\epsilon_\circ>0$. Suppose that $g\in \mathfrak F(\delta_\circ,1+\epsilon_\circ)$ and $z_0\in Q(0,\frac12)$. Then, there is a constant $c>0$, independent of $g$ and $z_0$, such that 
$[g]_{z_0}^{\epsilon}\in  \mathfrak F(\delta_\circ,1)$ provided that  $0<\ve<c$. 
\end{lemma}

Let us set $Q=Q(0,1)$. 
For $R\ge 1$ and a given $p\ge 1$,   we  define 
\[   \widetilde{\mathcal A}^p(R) =\sup\big\{ \| \mathcal E^\phi_{Q} f\|_{L^p(Q(0,R))} :  \|f\|_p\le 1, \phi\in \mathfrak F(\delta_0, 1+\epsilon_0)\big\} .   \] 
Then, we have the following which is a variant of Proposition \ref{pscaling}.

\begin{lemma}[Parabolic rescaling]
\label{pscaling-var}
Let $\delta_\circ$, $\epsilon_\circ >0$. Suppose $\phi\in \mathfrak F(\delta_0, 1+\epsilon_0)$, then there is an $r_\circ>0$ and $C>0$, independent of  $\phi$,  such that 
$$\Vert \mathcal E f \Vert_{L^p_{Q(0,R)}}\leq Cr^{2(n-1)-\frac{4n}{p}}\widetilde{\mathcal A}^{p}(R)\Vert f\Vert_p$$
whenever  $f$ is supported in  $Q(z_0,r)\in Q(0,1)$ with $r\le r_\circ$. 
\end{lemma}

This can be shown similarly  as in the proof of Lemma \ref{pscaling}  by making use of Lemma \ref{scaled}. So, we omit the proof. 
One may find a detailed argument  of similar nature in \cite{Lee}.

\begin{remark}
Complex analyticity assumption plays important roles in our overall argument. It has been used various steps, so
 it is not clear at the moment how to generalize the result to general 2-dimension surfaces without complex analyticity. 
\end{remark}

\begin{remark}\label{oddrmk}
The advantage of analyticity  is compensated by new difficulty which results from lack of curved property of the complex surfaces. 
Unlike the case of elliptic surfaces, restriction of  the surface to lower dimensional vector spaces does not necessarily   guarantee  
persistence of the transversality since the complex analytic quadratic function can be factorized. 
 For example, let $n=2k+1$ and  consider a complex surface given by a holomorphic function 
\[\phi(z_1,z_2)=z_1^2+z_2^2+\dots+z_{n-1}^2=(z_1+iz_2)(z_1-iz_2)+\dots+(z_{n-2}+iz_{n-1})(z_{n-2}-iz_{n-1}).\]
If we take $V_{k}=\{z: (z_1-iz_2)=c_1, \dots,  (z_{n-2}-iz_{n-1})=c_k\}$.  The restriction of the surface to $V_k$ does not have any curved property. 
\end{remark}

\subsection{Almost complex structure}
Now we consider slightly more general manifolds of codimension 2. We recall the definition of almost complex structure. 
\begin{definition}
Let $V$ be a real vector space of dimension $2n$. The automorphism $J:V\longrightarrow V$ is called  an almost complex structure if $J^2=-id$.
\end{definition}
This  gives a complex structure on $V$. We can easily check that $V$ becomes a complex vector space with $i\cdot v:=J(v)$ for all $v\in V$.  Suppose $\left( V,\langle \cdot,\cdot \rangle \right)$ is an inner product space. We say the almost complex structure $J$ is compatible with the inner product if $\langle J(u), J(v)\rangle=\langle u,v\rangle$ holds for all $u,v \in V$. In what follows, by almost complex structure we mean an almost complex structure which is compatible with the inner product.

Consider $V=\R^{2n-2}$ with usual inner product. We may regard  an almost complex structure $J$ as a matrix. Then, by using simple linear algebra, $J$ is an almost complex structure  (compatible with the inner product)  if and only if $J$ is a skew-symmetric orthogonal matrix. Using this $J$, we define our codimension 2 surface by a graph of $\left( \phi_1 (z), \phi_2 (z) \right)$ for $z\in \R^{2n-2}$ which satisfies
\begin{equation}\label{ac}
\nabla\phi_2 (z)=J\nabla\phi_1(z).
\end{equation}
We call these kind of manifolds "almost complex hypersurfaces".
For example, if
\[J=J_0\coloneqq
\begin{pmatrix}
P & 0  &\hdots  & 0  \\
 0 & P&\hdots  & 0\\
 \vdots & \vdots  &\ddots & \vdots \\
 0 & 0&\hdots  & P
\end{pmatrix}
\]
where 
$P=\begin{pmatrix}
0 & -1\\
1 & 0
\end{pmatrix}$, then \eqref{ac} is just the Cauchy-Riemann equation and the case corresponds to the complex analytic case. 

Since the minimal polynomial of $P$ is $x^2+1$, it is diagonalizable. Moreover, the eigenspaces of $i$ and $-i$ have the same dimension. Using this fact, and by suitable linear changes, we indeed know that almost complex hypersurfaces are essentially complex hypersurfaces. In other words, imposing different almost complex structure just determines how we identify $\R^{2n-2}$ with $\C^{n-1}$. Let $v_1, v_2, \cdots, v_{n-1}$ be eigenvectors with respect to the eigenvalue $i$. Since $J$ is a real matrix, $\overline{v_1}, \overline{v_2}, \cdots , \overline{v_{n-1}}\in\R^{2n-2}$ are eigenvectors with respect to the eigenvalue $-i$. Thus, the following $(2n-2)\times (2n-2)$ matrix is real and invertible:
$$L=\Big(
\frac{v_1+\overline{v_1}}{2}, \frac{v_1-\overline{v_1}}{2i},  \cdots,  \frac{v_{n-1}+\overline{v_{n-1}}}{2},  \frac{v_{n-1}-\overline{v_{n-1}}}{2i}
\Big).$$
Clearly, $ L^{-1}JL=J_0.$
This means that we can always reduce the restriction problem for the almost complex hypersurfaces to that of the complex hypersurfaces.

{\bf Acknowledgement.}  J. Lee was  supported by NRF-2017H1A2A1043158 and S. Lee  was supported in part by NRF-2018R1A2B2006298.  

\end{document}